\documentclass[12pt]{amsart}
\usepackage{amsmath,amssymb}

\usepackage{epsfig}  		
\usepackage[utf8]{inputenc}
\usepackage{pgf,tikz}
\usepackage{mathrsfs}
\usetikzlibrary{arrows}
\usepackage{graphicx}
\usepackage{mathtools}
\usepackage{amsmath}
\usepackage{amsthm}
\usepackage{amscd}
\usepackage{amsfonts}
\usepackage{fancyhdr}
\usepackage{fullpage}
\usepackage{epsfig}  		
\usepackage{epic,eepic}

\usepackage{hyperref}
\usepackage[english]{babel}
\usepackage{caption}
\usepackage{graphicx}
\usepackage{float}
\usepackage{xcolor}

\usepackage{pgf,tikz}
\usepackage{mathrsfs}
\usetikzlibrary{arrows}

\theoremstyle{proof}
\theoremstyle{plain}
\newtheorem{theorem}{Theorem}[section]
\newtheorem{lem}[theorem]{Lemma}
\newtheorem{prop}[theorem]{Proposition}
\newtheorem{col}[theorem]{Corollary}
\theoremstyle{definition}
\newtheorem{defn}[theorem]{Definition}

\newtheorem{example}[theorem]{Example}

\newtheorem{notation}[theorem]{Notation}

\newcommand{\XX}{{\mathcal{X}}}
\newcommand{\YY}{{\mathcal{Y}}}

\newcommand{\RR}{{\mathcal{R}}}
\newcommand{\D}{\Delta}
\newcommand{\st}{\ | \ }
\newcommand{\ver}{\mbox{Vert}}

\begin{document}

\title{Graded Betti numbers of path ideals of cycles and lines}

 \author{Ali Alilooee}
\address{University of Wisconsin-Stout, Department of Mathematics and Statistics,
Jarvis Hall-Science Wing,
Menomonie, WI, USA }
\email{ a-alilooeedolatabad@wiu.edu}

\author{Sara Faridi}
\address{Dalhousie University, Department of Mathematics, 6316 Coburg Rd.
Halifax, NS, Canada B3H 4R2 }
\email{faridi@mathstat.dal.ca}















\maketitle

\begin{abstract}

We use purely combinatorial arguments to give a formula to compute all graded Betti numbers of path ideals of paths and cycles.
As a consequence we can give new and short proofs for the known formulas of
regularity and projective dimensions of path ideals of paths. 

\end{abstract}



\section{Introduction}


Path complexes are simplicial complexes whose facets encode paths of a
fixed length in a graph. These simplicial complexes in turn correspond
to monomial ideals called ``path ideals''.  Path ideals of graphs were
first introduced by~\cite{Conca1999} in a different
algebraic context, but the study of algebraic invariants corresponding
to their minimal free resolutions has become popular, with works of~\cite{R.Bouchat2010} and~\cite{He2010}, and the authors~\cite{AF2012}.
The papers cited above give partial information on Betti numbers of
path ideals. In this paper we use purely combinatorial arguments based
on our results in~\cite{AF2012} to give an explicit formula for all the
graded Betti numbers of path ideals of paths and cycles. As a
consequence we can give new and short proofs for the known formulas of
regularity and projective dimensions of path ideals of path graphs.


\section{Preliminaries}





A \textbf{simplicial complex} on vertex set $\XX=\{x_1,\dots,x_n\}$ is
a collection $\D$ of subsets of $\XX$ such that $\{x_i\}\in \D$ for
all i, and if $F \in \D$ and $G \subset F$, then $G \in \D$. The
elements of $\D$ are called \textbf{faces} of $\D$ and the maximal
faces under inclusion are called \textbf{facets} of $\D$. We denote
the simplicial complex $\D$ with facets $F_1,\dots,F_s$ by $\langle
F_1,\dots,F_s\rangle$. We call $\{F_1,\dots,F_s\}$ the facet set of
$\D$ and is denoted by $F(\D)$. The set $\XX$ is the vertex set of
$\D$ and is denoted by $\ver(\D)$.  A \textbf{subcollection} of $\D$
is a simplicial complex whose facet set is a subset of the facet set
of $\D$. For $\YY\subseteq\XX$, an \textbf{induced subcollection} of
$\D $ on $\YY$, denoted by $\D_{\YY}$, is the simplicial complex whose
vertex set is a subset of $\YY$ and facet set is $\{F\in F(\D) \st F
\subseteq {\YY}\}.$
If $F$ is a face of $\D=\langle F_1,\dots,F_s\rangle$ and $\XX=\ver(\D)$, the
\textbf{complements} of a facet $F_i$ and of the simplicial complex $\D$ are 
\begin{eqnarray*}
(F_i)_\XX^c=\XX\setminus {F_i} & \mbox{and}& \D_\XX^c=\langle (F_1)^c_\XX,\dots,(F_s)^c_\XX \rangle.
\end{eqnarray*}
Note that if $\YY \subsetneqq \ver (\D)$, then
$(\D_{\XX}^{c})_\YY=(\D_\YY)^{c}_\YY$.

From now on we assume that $R=K\left[x_1,\dots,x_n\right]$ is a
polynomial ring over a field $K$. Suppose that $I$ is an ideal in $R$
minimally generated by square-free monomials $M_1,\ldots,M_s$. The
\textbf{facet complex} $\D(I)$ associated to $I$ has vertex set
$\{x_1,\dots,x_n\}$ and is defined as $\D(I)=\langle F_1,\ldots,F_s
\rangle \mbox{ where } F_i=\{x_j \st x_j | M_i,\ 1\leq j\leq n\}, \ 1
\leq i \leq s.$ Conversely if $\D$ is a simplicial complex with
vertices labeled $x_1,\ldots,x_n$, the \textbf{facet ideal} of $\D$ is
defined as $I(\D)=( \prod_{x \in F}x \st \ F \mbox{ is a facet of} \D).$

Given a  homogeneous ideal $I$ of the polynomial ring $R$ there exists a \textbf{graded minimal finite free resolution}
$$0\rightarrow {\displaystyle
   \bigoplus_{d}}R(-d)^{\beta_{p,d}}\rightarrow\cdots{\displaystyle
   \rightarrow\bigoplus_{d}}R(-d)^{\beta_{1,d}}\rightarrow
 R\rightarrow R /I \rightarrow 0$$ of $R /I $ in which $R(-d)$ denotes
 the graded free module obtained by shifting the degrees of elements
 in $R$ by $d$. The numbers $\beta_{i,d}$ are the $i$-th
 $\mathbb{N}$-\textbf{graded Betti numbers} of degree $d$ of $R /I$,
 and are independent of the choice of graded minimal finite free
 resolution.












By a reformulation of Hochster's formula in~\cite[Theorem~2.8]{AF2012},
to compute Betti numbers we only need to consider induced
subcollections $\Gamma=\D_\YY$ of a simplicial complex $\D$ with
$\YY=\ver (\Gamma)$.


\section{Path complexes and runs}


\begin{defn}

Let $G=(\XX,E)$ be a finite simple graph and let $t$ be an integer such
that $t\geq 2$.  If $x$ and $y$ are two vertices of $G$, a
\textbf{path} (or path) of length $(t-1)$ from $x$ to $y$ is a
sequence of distinct vertices $x=x_{i_1},\dots, x_{i_{t}}=y$ of $G$ such that
$\{x_{i_{j}},x_{i_{j+1}}\}\in E$ for all $j= 1,2,\dots,t-1$.
We define the {\bf path ideal} of $G$, denoted by $I_t(G)$ to be the
ideal of $K[x_1,\dots,x_n]$ generated by the monomials of the form
$x_{i_1}x_{i_2}\dots x_{i_t}$ where $x_{i_1},x_{i_2},\dots,x_{i_t}$ is
a path in $G$.
The facet complex of $I_t(G)$, denoted by $\D_t(G)$, is called the {\bf path complex} of the graph $G$.
\end{defn}
Two special cases that we will be considering in this paper are when
$G$ is a {\bf cycle} $C_n$, or a {\bf path graph} (or {\bf path}) $L_n$ on vertices
$\{x_1,\dots,x_n\}$.
$$C_n=\langle x_1x_2,\ldots,x_{n-1}x_n,x_nx_1\rangle \mbox{\ and \ }
L_n=\langle x_1x_2,\ldots,x_{n-1}x_n\rangle.$$
\begin{example} Consider the cycle $C_5$ with vertex set $\XX=\{x_1,\dots,x_5\}$ Then 
$$I_4(C_5)= (x_1x_2x_3x_4, \\ x_2x_3x_4x_5, x_3x_4x_5x_1, x_4x_5x_1x_2, x_5x_1x_2x_3).$$
\end{example}

\begin{notation} Let $i$ and $n$ be two positive integers.
For (a set of) labeled objects we use the notation $\mod n$ to denote
$$x_i \mod n \ =\{x_j \st 1\leq j \leq n, i\equiv j \mod n\}$$
and
$$\{x_{u_1},x_{u_2},\dots,x_{u_t}\} \mod n\  =\{x_{u_j}\mod n \st j=1,2,\dots,n\}.$$
\end{notation}
Let $C_n$ be a cycle on vertex set $\XX=\{x_1,\dots,x_n\}$ and $t< n$. The {\bf standard labeling} of the facets of 
$\D_t(C_n)$ is as follows. 

We let $\D_t(C_n)=\langle F_1,\dots,F_n\rangle$ where $F_i=\{x_i,x_{i+1},\dots,x_{i+t-1}\}\mod n$ for 
all $1\leq  i \leq n$. Since for each $1\leq i \leq n$ we have $$\begin{array}{llll}
   F_{i+1}\setminus F_{i}=\{x_{t+i}\}&\mbox{and}& F_{i}\setminus
   F_{i+1}=\{x_{i}\}&\mod n,
\end{array}$$
it follows that $\begin{array}{lllll} \left|F_i\setminus
  F_{i+1}\right|=1&\mbox{and}& \left|F_{i+1}\setminus
  F_{i}\right|=1&\mod n&\mbox{for all $1\leq i\leq
    n-1$}.
\end{array}$

\begin{defn}\label{defn:defn3.5}
 Given an integer $t$, we define a
{\bf run} to be the path complex of a path graph. A run which has $p$
facets is called a \textbf{run of length $p$} and corresponds to
$\D_t(L_{p+t-1})$. Therefore, a run of length $p$ has $p+t-1$ vertices.
\end{defn}

\begin{example} Consider the cycle $C_7$ on vertex set
$\XX=\{x_1,\dots x_7\}$ and the simplicial complex $\D_4(C_7)$. The following induced subcollections are two runs in  $\D_4(C_7)$ 
$$\begin{array}{lll} \D_1&=&\langle
   \{x_1,x_2,x_3,x_4\},\{x_2,x_3,x_4,x_5\}\rangle\\ \D_2&=&\langle
    \{x_1,x_2,x_6,x_7\},\{x_1,x_2,x_3,x_7\},\{x_1,x_2,x_3,x_4\}\rangle.
\end{array}$$
\end{example}
In~\cite{AF2012} we show that every induced subcollection of the path complex of a cycle is a disjoint union of runs
\cite[Proposition~3.6]{AF2012}, and that two induced subcollections of the path complex of a cycle composed of the same number of runs of the
same lengths are homeomorphic \cite[Lemma~3.2.9]{AThesis2015}. Then all the information we need to compute the homologies of induced
subcollections of $\D_t(C_n)$ depends on the number and the lengths of
the runs.
\begin{defn}\label{defn:e(s_1)} For a fixed integer $t \geq 2$,
let the pure $(t-1)$-dimensional simplicial complex $\Gamma=\langle
F_1,\ldots,F_s\rangle$ be a disjoint union of runs of length
$s_1,\ldots,s_r$. Then the sequence of positive integers
$s_1,\ldots,s_r$ is called a \textbf{run sequence} on
$\YY=\ver(\Gamma)$, and we use the notation
$$E(s_1,\dots,s_r)=\Gamma^c_\YY=\langle(F_1)_{{\YY}}^{c},\dots,
(F_s)_{\YY}^{c}\rangle.$$
\end{defn}


\section{Graded Betti numbers of path ideals}


We focus on Betti numbers of degree less than $n$, as those of degree $n$ were computed in~\cite{AF2012}. By~\cite[Theorem~2.8]{AF2012} we need to
count induced subcollections. 

\begin{defn}\label{d:eligible} Let $i$ and $j$ be positive integers.
We call an induced subcollection $\Gamma$ of $\D_t(C_n)$ an {\bf
$(i,j)$-eligible subcollection} of $\D_t(C_n)$ if $\Gamma$ is
composed of disjoint runs of lengths
 \begin{eqnarray} (t+1)p_1+1,\dots, (t+1)p_{\alpha}+1, (t+1)q_1+2,
  \ldots, (t+1)q_{\beta}+2\label{eqn:length1}
 \end{eqnarray}
for nonnegative integers $\alpha, \beta,p_1,p_2,\dots,p_{\alpha},q_1,q_2,\dots,q_{\beta}$, which satisfy the
following conditions
$$\begin{array}{lll} j&=&(t+1)(P+Q)+t(\alpha+\beta)+\beta
 \\ i&=&2(P+Q)+2\beta+\alpha,
  \end{array}$$
where $P=\sum_{i=1}^{\alpha} p_i$ and $Q=\sum_{i=1}^{\beta} q_i$.
\end{defn}
Eligible subcollections count the graded Betti numbers.
\begin{theorem}[~\cite{AF2012} Theorem~5.3]\label{lem:lem15} Let $I=I(\Lambda)$ be the facet ideal of an induced subcollection $\Lambda$ of
$\D_t(C_n)$. Suppose that $i$ and $j$ are integers with $i\leq j<n$. Then the ${\mathbb N}$-graded Betti number $\beta_{i,j}(R/I)$ is the number
of $(i,j)$-eligible subcollections of $\Lambda$.
\end{theorem}

The following corollary is a special case of Theorem~\ref{lem:lem15}.
\begin{col}\label{col:col16} Let $I=I(\Lambda)$ be the facet ideal of an induced subcollection $\Lambda$ of $\D_t(C_n)$. Then for every $i$,
$\beta_{i,ti}(R/I)$, is the number of induced subcollections of $\Lambda$ which are composed of $i$ runs of length 1.
\end{col}
\begin{proof} From Theorem~\ref{lem:lem15} we have $\beta_{i,ti}(R/I)$ is the number of $(i,ti)$-eligible subcollections of $\Lambda$. With
 notation as in Definition~\ref{d:eligible} we have
      $$\left\{\begin{array}{ll}
     ti=(t+1)(P+Q)+t(\alpha+\beta)+\beta&\\
     i=2(P+Q)+(\alpha+\beta)+\beta&\Rightarrow
     ti=2t(P+Q)+t(\alpha+\beta)+t\beta\\
       \end{array} \right.$$
Putting the two equations for $ti$ together, we conclude that $(t-1)(P+Q+\beta )=0$. But $\beta$, $P$, $Q\geq 0$ and $t\geq 2$,
so we must have $$\beta=P=Q=0 \Rightarrow p_1=p_2=\dots=p_{\alpha}=0.$$ So $\alpha=i$ and $\Gamma$ is composed of $i$ runs of length one.
  \end{proof}
Theorem~\ref{lem:lem15} holds in particular for $\Lambda=\D_t(L_{m})$ and $\Lambda=\D_t(C_{n})$ for any integers
 $m,n$. Our next statement is in a sense a converse to Theorem~\ref{lem:lem15}.
\begin{prop}\label{prop:tree}Let $t$ and $n$ be integers such that $2 \leq t \leq n$ and $I=I(\Lambda)$ be the facet ideal of $\Lambda$
 where $\Lambda$ is an induced subcollection of $\D_t(C_n)$. Then for each $i,j\in\mathbb{N}$, if $i <n$ and 
$\beta_{i,j}(R/I)\neq0$, there exist nonnegative integers $\ell,d$ such that $i\leq d<n$ and 
$$\left\{\begin{array}{lll} i&=&\ell+d\\ j&=&t\ell+d
\end{array}\right.$$
\end{prop}
\begin{proof} From Theorem~\ref{lem:lem15} we know $\beta_{i,j}$ is equal to the number of $(i,j)$-eligible subcollections of $\Lambda$,
 where with notation as in Definition~\ref{d:eligible} we have
       $$\left\{\begin{array}{lcr}
     j=(t+1)(P+Q)+t(\alpha+\beta)+\beta
      \\ i=2(P+Q)+(\alpha+\beta)+\beta.
       \end{array}\right.$$
 It follows that
     \begin{eqnarray}
       j-i=(t-1)(P+Q+\alpha+\beta)&\mbox{and}&
       ti-j=(t-1)(P+Q+\beta)\label{eqn:ell}.
      \end{eqnarray}
We now show that there exist positive integers $\ell,d$ such that
     $i=\ell+d$ and $j=t\ell+d$. $$\begin{array}{lll}
       \left\{\begin{array}{lcr} i=\ell+d \\ j=t\ell+d
      \end{array}\right. \Rightarrow
      \begin{array}{lll} \ell= \displaystyle \frac{j-i}{t-1}&
     \mbox{and}& d=\displaystyle \frac{ti-j}{t-1}
    \end{array}.
     \end{array}$$
From (\ref{eqn:ell}) we can see that $i$ and $j$ as described above
   are nonnegative integers.
   \end{proof}
Corollary~\ref{col:col16} tells us that to compute Betti numbers $\beta_{i,ti}$ of induced subcollections of $\D_t(C_n)$ we need to count the number
 of its induced subcollections which consist of disjoint runs of length one. The next few pages are dedicated to counting
 such subcollections.  We use some combinatorial methods to  generalize a helpful formula which can be found in Stanley's
 book~\cite[page~73]{R.P.Stanley1994}.
   \begin{lem}\label{lem:lem5.6} Consider a collection of $n$
points arranged on a line. The number of ways to color $k$ points,
when there are at least $t$ uncolored points on the line between each
colored point is
$${{n-(k-1)t}\choose{k}}.$$
   \end{lem}
\begin{proof} First label the points by $1,2,\dots,n$ from left to right, and let $a_1<a_2<\dots<a_k$ be the colored  points. For $1\leq i\leq k-1$, we define $x_i$ to be the
  number of points, including $a_{i}$, which are between $a_i$
        and $a_{i+1}$, and $x_0$ to be the number of points which
        exist before $a_1$, and $x_k$  the number of points,
        including $a_k$, which are after $a_k$.
        $$\begin{array}{llllll}
        \overbrace{\cdots}^{x_0} &\overbrace{\bullet\  \cdots}^{x_1} &
         \overbrace{\bullet\  \cdots}^{x_2} &
         {\bullet}\ \cdots &\overbrace{\bullet\ \cdots}^{x_{k-1}} &
         \overbrace{\bullet\ \cdots}^{x_{k}}\\
      1&a_1&a_2&a_3&a_{k-1}&a_k\ \ n\\
        \end{array}$$
        If we consider the sequence $x_0,x_1,\dots,x_k$ it is not
        difficult to see that there is a one to one correspondence
        between the positive integer solutions of the following
        equation and the ways of coloring $k$ points of $n$ points on
        a line with at least $t$ uncolored points between each two
        colored points.
      \begin{eqnarray*} x_0+x_1+\dots+x_k=n&\mbox{$x_0\geq 0$, $x_i > t$,
      for $1\leq i \leq k-1 $, and $x_k\geq 1$}.
      \end{eqnarray*}
      So we only need to find the number of positive integer
       solutions of this equation. Consider the following
     equation $$(x_0+1)+(x_1-t)+\dots+(x_{k-1}-t)+x_k=n-(k-1)t+1$$
       where $x_0+1\geq 1$, $x_i-t \geq 1$, for $i=0\dots,k-1$ and
     $x_k\geq 1$. The number of positive integer solution of this
       equation is (see for example~\cite {R.P.Grimaldi2003} page
       29) $${{n-(k-1)t}\choose{k}}.$$
      \end{proof}


 \begin{col}\label{col:mycol} Let $C_n$ be a graph cycle
and with the standard labeling let $\Gamma$ be a proper subcollection
of $\D_t(C_n)$ with $k$ facets $F_a, \ldots,F_{a+k-1} \mod n$. The
number of induced subcollections of $\Gamma$ which are composed of $m$
runs of length one is $${k-(m-1)t\choose m}.$$
\end{col}

   \begin{proof} To compute the number of induced
      subcollections of $\Gamma$ which are composed of $m$ runs of
      length one, it is enough to consider the facets $F_a,\ldots,F_{a+k-1}$ as points arranged on a line and compute the
      number of ways which we can color $m$ points of these $k$
      arranged points with at least $t$ uncolored points between each
      two consecutive colored points. Therefore, by
      Lemma~\ref{lem:lem5.6} we have the number of induced
      subcollections of $\Gamma$ which are composed of $m$ runs of
      length one is ${k-(m-1)t\choose m}.$
     \end{proof}

\begin{prop}\label{lem:lem1} Let $C_n$ be a graph cycle with vertex
 set $\XX=\{x_1,\dots,x_n\}$. The number of induced subcollections of
 $\D_t(C_n)$ which are composed of $m$ runs of length one
 is $$\frac{n}{n-mt}{n-mt \choose m}.$$
\end{prop}
    \begin{proof} Recall that $\D_t(C_n)=\langle F_1,\dots,F_n\rangle$
    with standard labeling.  First we compute the number of induced
     subcollections of $\D_t(C_{n})$ which consist of $m$ runs of
   length one and do not contain the vertex $x_n$. There are $t$
    facets of $\D_t(C_n)$ which contain $x_n$, the remaining facets
    are $F_1,\dots,F_{n-t}$, and so by Corollary~\ref{col:mycol} the
     number we are looking for is
      \begin{eqnarray} {n-t-(m-1)t \choose m}
        = {n-mt \choose m}.\label{eqn:number}
      \end{eqnarray}
    Now we are going to compute the number of induced subcollections
     $\Gamma$ which consist of $m$ runs of length one and include
    $x_n$. We have $t$ facets which contain $x_n$, they are
     $F_{n-t+1}\dots,F_n$. Each such $\Gamma$ will contain one $F_i\in
    \{F_{n-t+1}\dots,F_n\}$ as the run containing $x_n$, and $m-1$
    other runs of length one which have to be chosen so that they are
     disjoint from $F_i$. So we are looking for $m-1$ runs of length
    one in the subcollection $\Gamma'=\langle
     F_{i+t},\ldots,F_{i-t}\rangle \mod n$. The subcollection
     $\Gamma'$ has $n-2t-1$ facets, so by Corollary~\ref{col:mycol} it
     has $${n-2t-1-(m-2)t \choose m-1}={n-mt-1 \choose m-1}$$ induced
     subcollections that consist of runs of length one.
     Putting this together with the number of ways to choose $F_i$ and
     with (\ref{eqn:number}) we conclude that the number of induced
     subcollections of $\D_t(C_n)$ which are composed of $m$ runs of
     length one is $$t{n-mt-1 \choose m-1}+ {n-mt \choose m}=
  \frac{n}{n-mt} {n-mt \choose m}.$$
\end{proof}

We apply these counting facts to find Betti numbers in specific
degrees.

\begin{col}\label{col:col4.2} Let $n\geq 2$ and $t$ be an integer such
that $2\leq t \leq n$. Then for the cycle $C_n$ we
  have
  $$\beta_{i,ti}(R/I_t(C_n))=\frac{n}{n-ti}{ n-ti \choose i }.$$
\end{col}

\begin{proof} From Corollary~\ref{col:col16} we have $\beta_{i,ti}(R/I_t(C_n))$
     in each of the three cases is the number of induced
     subcollections of $\D_t(C_n)$ which are composed of $i$ runs of
     length 1. The formula now follows from Proposition~\ref{lem:lem1}.
     \end{proof}

The following Lemma is the core of our counting later on in this section.
\begin{lem}\label{lem:addingfacet}
 Let $\D_t(C_n)=\langle F_1,F_2,\dots, F_n \rangle$, $2\leq t \leq n$,
 be the standard labeling of the path complex of a cycle $C_n$ on
 vertex set $\XX=\{x_1,\ldots,x_n\}$. Let $i$ be a positive integer
 and $\Gamma=\langle F_{c_1},F_{c_2},\dots,F_{c_i}\rangle$ be an
 induced subcollection of $\D_{t}(C_n)$ consisting of $i$ runs of
 length 1, with $1\leq c_1<c_2<\dots<c_i\leq n$.  Suppose that $\Sigma$ is
 the induced subcollection on $\ver(\Gamma)\cup \{x_{c_u+t}\}$ for
some $1\leq u \leq i$. Then
$$|\Sigma|= \left\{
\begin{array}{lll}
|\Gamma|+ t & u <i  \ \mbox{ and}& c_{u+1}=c_u+t+1\\
|\Gamma|+ 1 & u =i  \ \mbox{ or}& c_{u+1}>c_u+t+1
\end{array}\right.$$
\end{lem}
    \begin{proof} Since  $\Gamma$ consists of runs of length one and
    each $F_{c_u}=\{x_{c_u},x_{c_u+1},\dots,x_{c_u+t-1}\}$ we must
   have $c_{u+1}>c_u+t \mod n$ for $u\in \{1,2,\dots,i-1\}$.  There
    are two ways that $x_{c_u+t}$ could add facets to $\Gamma$ to
    obtain $\Sigma$.
\begin{enumerate}
     \item If $c_{u+1}=c_u+t+1$ then
       $F_{c_u},F_{c_u+1},\dots,F_{c_u+t+1}=F_{c_{u+1}}\in\Sigma$ or in
      other words, we have added $t$ new facets to $\Gamma$.
      \item If $c_{u+1}>c_u+t+1$ or $u=i$ then $F_{c_u+1}\in \Sigma$,
        and therefore one new facet is added to $\Gamma$.
    \end{enumerate}
 \end{proof}

The following propositions, which generalize Lemma~7.4.22 in~\cite{Jacques2004}, will help us to compute the remaining Betti
numbers.

\begin{prop}\label{lem:lem4.3} Let $\D_t(C_n)=\langle F_1,F_2,\dots,
 F_n \rangle$, $2\leq t \leq n$, be the standard labeling of the path
 complex of a cycle $C_n$ on vertex set $\XX=\{x_1,\ldots,x_n\}$. Also
 let $i$, $j$ be positive integers such that $j\leq i$ and
 $\Gamma=\langle F_{c_1},F_{c_2},\dots,F_{c_i}\rangle$ be an induced
 subcollection of $\D_{t}(C_n)$ consisting of $i$ runs of length 1,
 with $1\leq c_1<c_2<\dots<c_i\leq n$. Suppose that $W=\ver(\Gamma)\cup A \subsetneq \XX$ for
 some subset $A$ of $\{x_{c_1+t},\dots,x_{c_i+t}\} \mod n$ with
 $|A|=j$. Then the induced subcollection $\Sigma$ of $\D_t(C_n)$ on
 $W$ is an $(i+j,ti+j)$-eligible subcollection.
\end{prop}

  \begin{proof} Since  $\Gamma$ consists of runs of length one and
    each $F_{c_u}=\{x_{c_u},x_{c_u+1},\dots,x_{c_u+t-1}\}$ we must
    have $c_{u+1}>c_u+t \mod n$ for $u\in \{1,2,\dots,i-1\}$.
     The runs (or connected components) of $\Sigma$ are of the form
     $\Sigma^{\prime}=\Sigma_{U}$ where $U\subseteq W$, and can have
     one of the following possible forms.
    \renewcommand{\theenumi}{\alph{enumi}}
     \renewcommand{\theenumii}{\Roman{enumii}}
    \begin {enumerate}
    \item For some $a\leq i$: $$U=F_{c_a},$$ and therefore
      $\Sigma^{\prime}=\langle F_{c_a}\rangle$ is a run of length 1.
   \item For some $a\leq i$: $$U=F_{c_a}\cup\{x_{{c_a}+t}\},$$ and
    therefore $c_{a+1}> c_a+t+1$, so from Lemma~\ref{lem:addingfacet} we have
    $\Sigma^{\prime}=\langle F_{c_a},F_{c_a+1}\rangle$ is a run of
    length 2.
    \item For some $a\leq i$: $$U=F_{c_a}\cup
     F_{c_{a+1}}\cup\dots\cup
       F_{c_{a+r}}\cup\{x_{c_a+t},x_{c_{a+1}+t},\dots,x_{c_{a+r-1}+t}\}\hspace{.1
         in} \mod n$$ and $F_{c_{a+j}}=F_{{c_a}+j(t+1)}$ for
       $j=0,1,\dots,r$ and $r\geq 1$.  Then from
       Lemma~\ref{lem:addingfacet} above we know $\Sigma^{\prime}$ is
       a run of length $r+1+tr=(t+1)r+1$.
     \item For some $a\leq i$: $$U=F_{c_a}\cup
       F_{c_{a+1}}\cup\dots\cup
      F_{c_{a+r}}\cup\{x_{c_a+t},x_{c_{a+1}+t},\dots,x_{c_{a+r}+t}\} \hspace{.1
         in} \mod n$$ and $F_{c_{a+j}}=F_{{c_a}+j(t+1)}$ for
       $j=0,1,\dots,r$ and $r\geq 1$, and $c_{a+r+1}>c_{a+r}+t+1$ or
       ${a+r}=i$.  Then from Lemma~\ref{lem:addingfacet} we have
       $\Sigma^{\prime}$ is a run of length $r+1+tr+1=(t+1)r+2$.
   \end{enumerate}
   So we have shown that $\Sigma$ consists of runs of length $1$ and
    $2$ $\mod t+1$.

Suppose that the runs in $\Sigma$ are of the form
    described in (\ref{eqn:length1}).  By
    Definition~\ref{defn:defn3.5} we have
    $$\begin{array}{ll} |\ver
     (\Sigma)|&=(t+1)p_1+t+\dots+(t+1)p_\alpha+t+(t+1)q_1+t+1+\dots+(t+1)
       q_{\beta}+t+1~\vspace{.1in}\\
      &=(t+1)P+t\alpha+(t+1)Q+t\beta+\beta ~\vspace{.1 in}\\
     &=(t+1)(P+Q)+t(\alpha+\beta)+\beta.
    \end{array}$$

 On the other hand by the definition of $\Sigma$ we know that,
  $\Sigma$ has $ti+j$ vertices and
   therefore $$ti+j=(t+1)(P+Q)+t(\alpha+\beta)+\beta.$$ It remains to
 show that $i+j=2(P+Q)+(\alpha+\beta)+\beta$.
   Note that if $j=0$ then $\beta=P=Q=0$ and hence
  \begin{eqnarray}
      j=0&\Longrightarrow& P+Q+\beta=0\label{eqn:neweqn}.
    \end{eqnarray}
  Moreover each vertex $x_{{c_v}+t}\in A$ either increases the length
 of a run in $\Gamma$ by one and hence increases $\beta$ (the number
 of runs of length 2 in $\Gamma$) by one, or increases the length of
a run by $t+1$, in which case $P+Q$ increases by 1. We can conclude
that if we add $j$ vertices to $\Gamma$, $P+Q+\beta$ increases by
   $j$. From this and (\ref{eqn:neweqn}) we have $j=P+Q+\beta$.
   Now we solve the following system
  $$\left\{ \begin{array}{rllll}
      ti+j&=&(t+1)(P+Q)+t(\alpha+\beta)+\beta &
       \Longrightarrow& ti=t(P+Q)+t(\alpha+\beta)\\
      j&=&P+Q+\beta&\Longrightarrow& i=P+Q+\alpha+\beta
     \end{array} \right.$$
   $$ \Longrightarrow \left\{\begin{array}{lll}
   i&=&P+Q+\alpha+\beta \\
      j&=&P+Q+\beta
    \end{array} \right.
     \Longrightarrow i+j=2(P+Q)+(\alpha+\beta)+\beta. $$
 \end{proof}
\begin{prop}\label{prop:newprop} Let $C_n$  be a cycle, $2\leq t \leq n$,
and $i$ and $j$ be positive integers. Suppose that $\Sigma$ is an
$(i+j,ti+j)$-eligible subcollection of $\D_t(C_n)$, $2\leq t \leq n$.
Then with notation as in Definition~\ref{d:eligible}, there exists a
unique induced subcollection $\Gamma$ of $\D_t(C_n)$ of the form
$\langle F_{c_1},F_{c_2},\dots,F_{c_i}\rangle$ with $1\leq
c_1<c_2<\dots<c_i\leq n$ consisting of $i$ runs of length $1$, and a
subset $A$ of $\{x_{c_1+t},\dots,x_{c_i+t}\}$ $\mod \ n$, with $|A|=j$
such that $\Sigma=\D_t(C_n)_{W}$ where  $W=\ver(\Gamma)\cup A.$
\bigskip

Moreover if $\RR=\langle F_{h},F_{h+1},\dots,F_{h+m}\rangle \mod n$ is a run
in $\Sigma$ with $|\RR|=2 \mod (t+1)$, then $F_{h+m} \notin \Gamma \mod n$.
\end{prop}
   \begin{proof} Suppose that $\Sigma$ consists of runs $R_1^{\prime},
    R_2^{\prime},\ldots,R_{\alpha+\beta}^{\prime}$
    where  for $k=1,2,\ldots,\alpha+\beta$
    $$\begin{array}{ll} R_{k}^{\prime}=\langle
      F_{h_k},F_{h_k+1},\dots,F_{h_k+m_k-1}\rangle & \mod n \vspace{.1 in}\\
      \ver(R_{k}^{\prime})=\{x_{h_k},x_{h_k+1},\dots,x_{h_k+m_k+t-2}\} &
    \mod n \vspace{.1 in}\\
     h_{k+1}\geq t+h_k+m_k & \mod n
    \end{array}$$
 and
 \begin{eqnarray}
     m_k=\left\{\begin{array}{lll}
           (t+1)p_k+1&\mbox{for}& k=1,2,\dots,\alpha\\
         (t+1)q_{k-\alpha}+2&\mbox{for}& k=\alpha+1,\alpha+2,\dots,\alpha+\beta.
       \end{array}
      \right. \label{eqn:mk}
   \end{eqnarray}
   For each $k$, we remove the following vertices from $\ver
   (R_{k}^{\prime})$
 \begin{eqnarray}
   \begin{array}{lll}
     x_{h_k+t},x_{h_k+2t+1},\dots,x_{h_k+p_k t+(p_k-1)}&\mod n&
      \mbox{ if } 1 \leq k \leq \alpha
      \mbox{ and } p_k\neq0  \\
      x_{h_k+t},x_{h_k+2t+1},\dots,x_{h_k+(q_{k-\alpha}+1) t+q_{k-\alpha}}& \mod n&
   \mbox{ if } \alpha+1\leq k \leq \alpha+\beta.
    \end{array}\label{eqn:deleted}
    \end{eqnarray}
   Let $\Gamma=\langle R_1,R_2,\dots R_{\alpha+\beta} \rangle$ be the induced subcollection on the remaining vertices of $\Sigma$, where
   \begin{eqnarray}
     R_k=\left\{
    \begin{array}{lll}
    \langle F_{h_k},F_{h_k+t+1},\dots,F_{h_k+(t+1)p_k}\rangle&
      \mod n &\mbox{for } 1 \leq k \leq \alpha\\
      \langle F_{h_k},F_{h_k+t+1},\dots,F_{h_k+(t+1)q_{k-\alpha}}\rangle&
      \mod n &\mbox{for }  \alpha+1\leq k \leq \alpha+\beta.
     \end{array}\right.\label{eqn:mohem}
    \end{eqnarray}
 In other words,$\mod n$,  $\Gamma$ has facets
    $$F_{h_1},F_{h_1+t+1},\dots,F_{h_1+(t+1)p_1},F_{h_2},F_{h_2+t+1},
     \dots,F_{h_2+(t+1)p_2},\dots,F_{h_{\alpha+\beta}},\dots,F_{h_{\alpha+\beta}+
       (t+1)q_{\beta}}. $$ It is clear that each $R_k$ consists of
     runs of length one. Since $\Gamma$ is a subcollection of
     $\Sigma$, no runs of $R_k$ and $R_{k^{\prime}}$ are connected to
     one another if $k\neq k^{\prime}$, and hence we can conclude that 
     $\Gamma$ is an induced subcollection of $\D_t(C_n)$ which is
     composed of runs of length one.
    From (\ref{eqn:mohem}) we have the number of runs of length 1 in
    $\Gamma$ (or the number of facets of $\Gamma$) is equal
    to $$(p_1+1)+(p_2+1)+\dots+(p_{\alpha}+1)+(q_1+1)+\dots+(q_{\beta}+1)=
   P+Q+\alpha+\beta=i.$$ Therefore, $\Gamma$ is an induced
  subcollection of $\D_t(C_n)$ which is composed of $i$ runs of
    length 1. We relabel the facets of $\Gamma$ as $\Gamma=\langle
   F_{c_1},\dots,F_{c_i}\rangle$.
   Now consider the following subset of
    $\{x_{c_1+t},\dots,x_{c_i+t}\}$ as $A$
     $$\bigcup_{k=1,p_k\neq
      0}^{\alpha}\{x_{h_k+t},x_{h_k+2t+1},\dots,x_{h_k+p_k
      t+(p_k-1)}\} \cup
    \bigcup_{k={\alpha+1}}^{\alpha+\beta}\{x_{h_k+t},x_{h_k+2t+1},
    \dots,x_{h_k+(q_{k-\alpha}+1)t+q_{k-\alpha}}\}$$ by
    (\ref{eqn:deleted}) we
    have: $$|A|=(p_1+p_2+\dots+p_{\alpha})+(q_1+1\dots+q_{\beta}+1)=P+Q+\beta=j.$$
    Then if we set $$W=(\bigcup_{h=1}^{i}F_{c_h})\cup A$$ we clearly
    have $\Sigma=(\D_t(C_n))_{W}$.
    This proves the existence of $\Gamma$. We now prove its
    uniqueness. Let $\Lambda=\langle
    F_{s_1},F_{s_2},\dots,F_{s_i}\rangle$ be an induced subcollection
    of $\D_t(C_n)$ which is composed of $i$ runs of length 1 such that
    $1\leq s_1<s_2<\dots<s_i\leq n$. Also let $B$ be a $j$- subset of
    the set
    $\begin{array}{lll}\{x_{s_1+t},x_{s_2+t},\dots,x_{s_i+t}\}&\mod
      n\end{array}$ such that
    \begin{equation}
      \Sigma=({\D_t(C_n)})_{\ver (\Lambda)\cup B}.\label{eqn:sum2}
     \end{equation}
    Suppose that $\Lambda=\langle S_1,S_2,\dots,S_{\alpha+\beta}\rangle$,
    such that for $k=1,2,\dots,\alpha+\beta$, $S_k$ is an induced
    subcollection of $R_k^{\prime}$ which consists of $y_k$ runs of
  length one. By (\ref{eqn:sum2}) we have $y_k\neq0$ for all
    $k$. Now we prove the following claims for each
    $k\in\{1,2,\dots,\alpha+\beta\}$.
  \renewcommand{\theenumi}{\alph{enumi}}
  \begin{enumerate}
   \item \emph{$F_{h_k}\in \Lambda$}.
     Suppose that $1\leq k \leq \alpha+\beta$. If $p_k=0$ we are clearly done, so consider the case $p_k\neq0$.

    Assume that $F_{h_k}\notin \Lambda$. Since $F_{h_k}$ is the only facet
     of $\Sigma$ which contains $x_{h_k}$ we can conclude that 
    $x_{h_k}\notin \ver (\Lambda)$. From (\ref{eqn:sum2}), it follows
     that $x_{h_k}\in \{x_{s_1+t},x_{s_2+t},\dots,x_{s_i+t}\}$, so
    \begin{eqnarray} x_{h_k}=x_{s_a+t} \mod n \mbox{ for some $a$}.\label{eqn:s-a}
    \end{eqnarray}
    On the other hand we know
    \begin{eqnarray*}
     F_{s_a}=\{x_{s_a},x_{s_a+1},\dots,x_{s_a+t-1}\} &\mod \ n\\
  F_{s_a+1}=\{x_{s_a+1},x_{s_a+2},\dots,x_{s_a+t}\} &\mod \ n.
    \end{eqnarray*}
    Since $R_k^{\prime}$ is an induced connected component of
    $\Sigma$, by (\ref{eqn:s-a}) we can conclude that $x_{h_k}\in
    F_{s_a+1}$ and $F_{s_a},F_{s_a+1}\in R_k^{\prime}$. However, we
    know $F_{h_k}$ is the only facet of $R_k^{\prime}$ which contains
    $x_{h_k}$ and so $F_{s_a+1}=F_{h_k}$ and then $\begin{array}{ll}
      s_a+1=h_k& \mod n\end{array}$. This and (\ref{eqn:s-a}) imply
     that $t=1 \mod n$, which contradicts our assumption $2\leq t
      \leq n$.

    $F_{u+t+1}\in R_k^{\prime}$, then $F_{u+t+1}\in S_k$.
   Assume that $F_{u+t+1}\notin S_k$ and $F_{u+t+1}\in
   R_k^{\prime}$. Let $$r_0=\min\{r: r>u , F_r\in S_k \mod n\}.$$
   Since $S_k$ consists of runs of length one we can conclude
   that $r_0\geq u+t+1$. Since $r_0\neq u+t+1$ we have $r_0\geq u+t+2$. But
   then $$x_{u+t+1}\notin \ver (\Lambda)\cup
  \{x_{s_1+t},x_{s_2+t},\dots,x_{s_i+t}\}$$ and therefore $x_{u+t+1}\notin
   \ver (\Sigma)$ which is a contradiction.
\end{enumerate}
 
 Now for each $k$, by (a) we have $F_{h_k}\in \Lambda$ and from
  repeated applications of (b) we find that
 $$F_{h_k+f(t+1)}\in S_k \hspace{.1 in}\mbox{ for }
   f=\left \{ \begin{array}{ll}
   1,2,\dots,p_k &  1\leq k \leq \alpha\\
 1,2,\dots,q_{k-\alpha}&  \alpha+1\leq k \leq \alpha+\beta.
  \end{array}\right.$$

  So $R_k \subseteq S_k$. On the other hand $S_k$ consists of runs of
  length one, so no other facet of $R'_k$ can be added to it, and
   therefore $S_k=R_k$ for all $k$. We conclude that $\Lambda=\Gamma$
  and we are therefore done.
   The last claim of the proposition is also apparent from this proof.
 \end{proof}

We are now ready to compute the remaining Betti numbers.
\begin{theorem}\label{theorem:theorem4}
 Let $n$, $i$, $j$ and $t$ be integers such that $n\geq 2$, $2\leq t
 \leq n$, and $ti+j < n$. Then 
\[
\beta_{i+j,ti+j}(R/I_t(C_n))=
\frac{n}{n-it}{i \choose j}{ n-it \choose
  i }
\]

\end{theorem}
  \begin{proof} If $I=I_t(C_n)$, from Theorem~\ref{lem:lem15}, $\beta_{i+j,ti+j}(R/I)$ is the
    number of $(i+j,ti+j)$-eligible subcollections of $\D_t(C_n)$.  
		Suppose that $\RR_{(i)}$ denotes the set of
      all induced subcollections of $\D_t(C_n)$ which are composed of
     $i$ runs of length one. By propositions~\ref{lem:lem4.3}
    and~\ref{prop:newprop} there exists a one to one correspondence
     between the set of all $(i+j,ti+j)$-eligible subcollections of
  $\D_t(C_n)$ and the set $$\RR_{(i)}\times {\left[i\right]\choose
      j}$$ where ${\left[i\right]\choose j}$ is the set of all
     $j$-subsets of a set with $i$ elements.
 By Corollary~\ref{col:col16} we have
   $|\RR_{(i)}|=\beta_{i,ti}$ and since
 $|{\left[i\right]\choose j}|={i \choose j}$ and so we apply
   Corollary~\ref{col:col4.2} to observe that
 \[
\beta_{i+j,ti+j}(R/I_t(C_n))={i \choose j}\beta_{i,ti}(R/I_t(C_n))= \frac{n}{n-ti}{i \choose j}{ n-ti
     \choose i }.
	\]

  \end{proof}

Finally, we put together Theorem~\ref{lem:lem15},
Proposition~\ref{prop:tree} and~\cite[Theorem~5.1 and
  Theorem~2.8]{AF2012}. Note that the case $t=2$ is the case of graphs
which appears in~\cite{Jacques2004}. Also note that
$\beta_{i,j}(R/I_t(C_n))=0$ for all $i\geq 1$ and $j>ti$, see for
example see for example~\cite[3.3.4]{Jacques2004}.

\begin{theorem}[{\bf Betti numbers of path ideals of cycles}]\label{t:maintheorem} Let $n$, $t$, $p$ and $d$  be integers such that
$n\geq 2$, $2\leq t \leq n$, $n=(t+1)p+d$, where $p\geq 0$, $0\leq d
 \leq t$. Then the $\mathbb{N}$-graded Betti numbers of the path ideal of the
 graph cycle $C_n$ are given by
$$\beta_{i ,j}(R/I_t(C_n))= \left \{
  \begin{array}{ll}
  t & j=n,\ d= 0,\   \displaystyle i=2\left (\frac{n}{t+1}\right )\\
   &\\
  1 & j=n,\ d\neq 0, \  \displaystyle i=2\left (\frac{n-d}{t+1}\right )+1 \\
    &\\
\displaystyle \frac{n}{{n-t\left(\frac{j-i}{t-1}\right)}}
  {\frac{j-i}{t-1}\choose \frac{ti-j}{t-1}}
 {n-t\left(\frac{j-i}{t-1}\right)\choose \frac{j-i}{t-1}} &
  \left \{\begin{array}{l} 
   j< n, \ i\leq j \leq ti, \mbox{ and }\\ \\
   \displaystyle 2p\geq \frac{2(j-i)}{t-1}\geq i 
  \end{array} \right.
  \\
&\\
  0&\mbox{otherwise.}
   \end{array} \right.
   $$
\end{theorem}
\begin{proof} 

We only need to make sure we have the correct conditions for the Betti
numbers to be nonzero.  When $j<n$, $\beta_{i,j}(R/I_t(C_n))\neq 0 \Longleftrightarrow$ 

$$\begin{array}{lll}
      & \Longleftrightarrow \left\{
       \begin{array}{l}
       \displaystyle\frac{j-i}{t-1}\geq \frac{ti-j}{t-1}~\vspace{.1in}\\
      \displaystyle n-\frac{t(j-i)}{t-1}\geq \frac{j-i}{t-1}
      \end{array}\right. 
      & \\
      &&  \\
      & \Longleftrightarrow \left\{
      \begin{array}{l}
       \displaystyle 2j\geq (t+1)i~\vspace{.1 in} \\
        \displaystyle n\geq \left(\frac{t+1}{t-1}\right)(j-i)
      \end{array}\right.  
    &\\
   &&  \\   
    &\Longleftrightarrow  \left\{
   \begin{array}{l}
       \displaystyle 2j\geq (t+1)i~\vspace{.1 in} \\
   \displaystyle (t+1)p+d\geq \left(\frac{t+1}{t-1}\right)(j-i)
   \end{array}\right. 
   & \\
   &&\\
     &\Longleftrightarrow \left\{
      \begin{array}{l}
       \displaystyle 2j\geq (t+1)i~\vspace{.1 in} \\
       \displaystyle p+\frac{d}{t+1}\geq \frac{j-i}{t-1}
     \end{array}\right. 
   &\\
   &&  \\
   &\Longleftrightarrow
   \displaystyle 2p\geq \frac{2(j-i)}{t-1}\geq i 
    & \mbox{ as } d<t+1 \\
\end{array}$$
   \end{proof}

\begin{theorem}[{\bf Betti numbers of path ideals of lines }]\label{newtheorem} Let $n$, $t$, $p$ and $d$  be integers such that
$n\geq 2$, $2\leq t \leq n$, $n=(t+1)p+d$, where $p\geq 0$, $0\leq d
 \leq t$. Then \renewcommand{\theenumi}{\roman{enumi}}
the $\mathbb{N}$-graded Betti numbers of the path ideal of the
 path graph $L_n$ are nonzero and equal to 
$$\beta_{i ,j}(R/I_t(L_n))=
\displaystyle{\frac{j-i}{t-1}\choose \frac{ti-j}{t-1}}
{n-t\left(\frac{j-i}{t-1}\right)\choose \frac{j-i}{t-1}} +
{\frac{j-i}{t-1}-1\choose \frac{ti-j}{t-1}}
{n-t\left(\frac{j-i}{t-1}\right)\choose \frac{j-i}{t-1}-1}$$
if and only if 
\begin{enumerate}
\item $j\leq n$ and  $i\leq j \leq ti$;
\item If $d<t$  then $ \displaystyle     p\geq \frac{j-i}{t-1} \geq i/2$ where both inequalities cannot be $=$ at the same time; 
\item If  $d=t$ then $\displaystyle (p+1)\geq \frac{j-i}{t-1} \geq i/2$ where both inequalities cannot be $=$ at the same time.
\end{enumerate}
\end{theorem}
  \begin{proof} We use induction on $n$. Suppose $n=2$
 so $t=2$. Then we have $I_t(L_n)=(x_1x_2)$ and it is clear that the
only nonzero Betti number of $R/I_t(L_n)$ is $\beta_{1,2}=1$. So the
assertion is clear. So suppose $n>2$. To continue we will use the
 following formula for each integer $N,R>1$
			\[
 \binom{N}{R}+\binom{N}{R-1}=\binom{N+1}{R}. 
\]
From~\cite{R.Bouchat2010} we have the recursive formula
$$\beta_{i,j}(R/I_t(L_n))=
\beta_{i,j}(R/I_t(L_{n-1}))+\beta_{i-1,j-t}(R/I_t(L_{n-(t+1)}))+\beta_{i-2,j-t-1}(R/I_t(L_{n-(t+1)}))$$
which using the induction hypothesis leads to the following calculation for
$\beta_{i,j}(R/I_t(L_n))$
$$\begin{array}{l}
\displaystyle\binom{\frac{j-i}{t-1}}{\frac{ti-j}{t-1}}
  \binom{n-1-t\left(\frac{j-i}{t-1}\right)}{\frac{j-i}{t-1}}
  +
\binom{\frac{j-i}{t-1}-1}{\frac{ti-j}{t-1}}
\binom{n-1-t\left(\frac{j-i}{t-1}\right)}{\frac{j-i}{t-1}-1} \\ \\ 
+\displaystyle\binom{\frac{j-t-i+1}{t-1}}{\frac{ti-t-j+t}{t-1}}
\binom{n-(t+1)-t\left(\frac{j-t-i+1}{t-1}\right)}{\frac{j-t-i+1}{t-1}}
+\displaystyle\binom{\frac{j-t-i+1}{t-1}-1}{\frac{ti-t-j+t}{t-1}}
\binom{n-(t+1)-t\left(\frac{j-t-i+1}{t-1}\right)}{\frac{j-t-i+1}{t-1}-1}\\ \\ 
+\displaystyle\binom{\frac{j-t-1-i+2}{t-1}}{\frac{ti-2t-j+t+1}{t-1}}
\binom{n-(t+1)-t\left(\frac{j-t-1-i+2}{t-1}\right)}{\frac{j-t-1-i+2}{t-1}} 
+\displaystyle\binom{\frac{j-t-1-i+2}{t-1}-1}{\frac{ti-2t-j+t+1}{t-1}}
\binom{n-(t+1)-t\left(\frac{j-t-1-i+2}{t-1}\right)}{\frac{j-t-1-i+2}{t-1}-1}.
\end{array}$$
For ease of writing, we set $A=\frac{j-i}{t-1}$ and
$B=\frac{ti-j}{t-1}$, so we have that
$$\begin{array}{l}
  \beta_{i,j}(R/I_t(L_n))=
\binom{A}{B}
\binom{n-1-tA}{A} 
+ \binom{A-1}{B}
\binom{n-1-tA}{A-1} \\ \\
+\binom{A-1}{B}\binom{n-1-tA}{A-1}
+\binom{A-2}{B}
\binom{n-1-tA}{A-2}
+
\binom{A-1}{B-1}
\binom{n-1-tA}{A-1}+\binom{A-2}{B-1}
\binom{n-1-tA}{A-2} \\ \\ 
=\left[\binom{A-1}{B}+
\binom{A-1}{B-1}
\right]\binom{n-1-tA}{A-1}
+\left[\binom{A-2}{B}+ 
\binom{A-2}{B-1}\right]\binom{n-1-tA}{A-2}
+\binom{A}{B}\binom{n-1-tA}{A}
+\binom{A-1}{B}\binom{n-1-tA}{A-1} \\ \\ 
=\binom{A}{B}\binom{n-1-tA}{A-1}+
\binom{A-1}{B}\binom{n-1-tA}{A-2}
+\binom{A}{B}\binom{n-1-tA}{A}
+\binom{A-1}{B}\binom{n-1-tA}{A-1}  \\ \\
=\binom{A}{B}\left[\binom{n-1-tA}{A-1}
+\binom{n-1-tA}{A}\right]
+\binom{A-1}{B}\left[\binom{n-1-tA}{A-2}+
\binom{n-1-tA}{A-1}\right]\\ \\
=\binom{A}{B}\binom{n-tA}{A}
+\binom{A-1}{B}\binom{n-tA}{A-1}.
\end{array}$$

Using the notation above, we see that 
$\beta_{i,j}(R/I_t(L_n))\neq 0$ if and only if
$$[B\leq A \mbox{ and } A \leq n-tA ] \mbox{ or }
[B\leq A-1 \mbox{ and } A-1 \leq n-tA ]$$
which is equivalent to saying that 
$$B\leq A \mbox{ and } A-1 \leq n-tA  \mbox{ where both $\geq$ cannot be $=$ at the same time}.$$
In other words $\beta_{i,j}(R/I_t(L_n))\neq 0$ if and only if 
$$ \begin{array}{ll}
  \displaystyle\frac{j-i}{t-1}\geq \frac{ti-j}{t-1} \mbox{ and }
       \displaystyle n-\frac{t(j-i)}{t-1}\geq \frac{j-i}{t-1}-1
      & \iff \\
     \displaystyle 2j\geq (t+1)i \mbox{ and }
     \displaystyle n+1\geq \left(\frac{t+1}{t-1}\right)(j-i) 
     & \iff \\
     \displaystyle 2j\geq (t+1)i \mbox{ and }
     \displaystyle (t+1)p+d+1\geq \left(\frac{t+1}{t-1}\right)(j-i)
     & \iff \\
     \displaystyle 2j\geq (t+1)i \mbox{ and }
    \displaystyle p+\frac{d+1}{t+1}\geq \frac{j-i}{t-1} &  \\

\end{array} $$
where in each line both $\geq$ cannot be $=$ at the same time. This is equivalent to 
$$   \left\{ 
   \begin{array}{ll}
   \displaystyle 2p\geq \frac{2(j-i)}{t-1}\geq i 
   & \mbox{ if } d<t \\
    \displaystyle 2(p+1)\geq \frac{2(j-i)}{t-1}\geq i 
   & \mbox{ if } d=t   
   \end{array}\right.$$
(note that $\frac{j-i}{t-1}$ 
from Theorem~\ref{lem:lem15} and Definition~\ref{d:eligible} is an integer) where both $\geq$ cannot be $=$ at the same time in the second line.
   \end{proof}

We can now easily derive the projective dimension and regularity of
path ideals of paths, which were known before. The projective
dimension of paths (Part i below) was computed in~\cite{He2010} using different methods. The case $t=2$ is the case
of graphs which appears in \cite{Jacques2004}.  Part ii of the
following Corollary reproves \cite[Theorem~5.3]{R.Bouchat2010} which
computes the Castelnuovo-Mumford regularity of path ideal of a path.
The case of cycles was done in~\cite{AF2012}.
\begin{col}[{\bf Projective dimension and regularity of path ideals of paths}]\label{c:pdr}
 Let $n$, $t$, $p$ and $d$ be integers such that $n\geq 2$, $2\leq t
 \leq n$, $n=(t+1)p+d$, where $p\geq 0$, $0\leq d \leq t$. Then
 \renewcommand{\theenumi}{\roman{enumi}}
\begin{enumerate}
\item The projective dimension of the path ideal of a path $L_n$ is given by
$$pd(R/I_t(L_n))=\left\{\begin{array}{ll}
2p & d\neq t~\vspace{.1 in}\\
2p+1&d=t\\
\end{array}\right.$$

\item The regularity of the path ideal of a path $L_n$ is given by
$$reg(R/I_t(L_n))=\left\{\begin{array}{ll}
p(t-1) & d<t\vspace{.1 in}\\
(p+1)(t-1) & d=t\\
\end{array}\right.$$
\end{enumerate}
\end{col}
 \begin{proof} \renewcommand{\theenumi}{\roman{enumi}}
   \begin{enumerate}
     \item By using Theorem~\ref{t:maintheorem} we know that if $\beta_{i,j}(R/I_t(L_n)\neq 0$ then  $i\leq 2p+1$ when $d=t$ and therefore $pd(R/I_t(L_n))\leq 2p+1$.    
       On the other hand by applying Theorem~\ref{t:maintheorem} we have
   $$\beta_{2p+1,n}(R/I_t(L_n))= \displaystyle{p+1\choose p} {p\choose
    p+1} + {p\choose p}{p\choose p}=1\neq 0.$$ Then we can conclude that $pd(R/I_t(L_n))=2p+1$.
Now we suppose that $d\neq t$. From Theorem~\ref{newtheorem} we can see
 that if $\beta_{i,j}(R/I_t(L_n))\neq 0$ then $2p\geq i$ and
  therefore $pd(R/I_t(L_n))\leq 2p$.  On the other hand, by applying
 Theorem~\ref{newtheorem}, we can see that
    \begin{equation*}
    \beta_{2p,p(t+1)}(R/I_t(L_n))= \displaystyle{p\choose p}
    {p+d\choose p} + {p-1\choose p}{p\choose p}={p+d\choose p}\neq 0.
  \end{equation*}
  Therefore $pd(R/I_t(L_n))\geq 2p$ and we have $pd(R/I_t(L_n))= 2p$.
 \item By definition, the regularity of a module $M$ is $\max \{j-i
   \st \beta_{i,j}(M)\neq 0\}$. By Theorem~\ref{newtheorem}, we
    know exactly when the graded Betti numbers of $R/I_t(L_n)$ are
    nonzero, and the formula follows directly.
    \end{enumerate}
  \end{proof}
\section*{Acknowledgement}
We gratefully acknowledge the helpful computer algebra systems CoCoA~\cite{CoCoA-5} and 
Macaulay2~\cite{Macaulay2}, without which our work would have been difficult or impossible. We also thank the referee whose comments improved the paper.

\bibliographystyle{alpha}
\bibliography{JAA-graded-path[1100]}

\end{document}